\documentclass[12pt,leqno,twoside]{article}

\usepackage{fancyhdr}
\usepackage{titlesec}
\usepackage{microtype}
\usepackage{amsmath}
\usepackage{amsthm}
\usepackage{amsfonts}
\usepackage{amssymb}
\usepackage{graphicx}
\usepackage{color}
\usepackage{enumerate}
\usepackage[utf8]{inputenc}
\usepackage{hyperref}

\DeclareMathOperator{\tb}{tb}
\DeclareMathOperator{\T}{T}

\newcommand{\Z}{\mathbb{Z}}
\newcommand{\Q}{\mathbb{Q}}

\newtheoremstyle{thm}{.5\baselineskip}{.5\baselineskip}{\itshape}{5mm}
{\scshape}{.}{2mm}{}

\theoremstyle{thm}
\newtheorem{theorem}{Theorem}[section]

\newtheorem{lemma}[theorem]{Lemma}

\newtheorem{definition}[theorem]{Definition}

\newtheoremstyle{rem}{.5\baselineskip}{.5\baselineskip}{}{5mm}
{\scshape}{.}{2mm}{}

\theoremstyle{rem}
\newtheorem{remark}[theorem]{Remark}
\newtheorem{example}[theorem]{Example}
\newtheorem{algorithm}[theorem]{Algorithm}

\renewcommand{\abstractname}{{\bf Abstract.\ }}
\renewenvironment{abstract}{\footnotesize\abstractname}

\makeatletter

\renewcommand{\section}{\@startsection{section}{1}{0pt}{\baselineskip} 
{.5\baselineskip}{\centering\normalfont\bfseries}}

\renewcommand{\subsection}{\@startsection{subsection}{1}{7mm}{\baselineskip}
{-3mm}{\bfseries}}

\renewcommand*{\@seccntformat}[1]{%
  \csname the#1\endcsname.\ }

\makeatother

\begin{document}

\pagestyle{fancy}

\thispagestyle{empty}

\begin{center}
{\bf\large COMPUTING THE THURSTON--BENNEQUIN INVARIANT IN OPEN BOOKS
}

\bigskip

{\begin{NoHyper}\small SEBASTIAN DURST$^1$, MARC KEGEL$^2$ and MIRKO KLUKAS$^3$\footnote{\parindent=6mm{\it Key words and phrases:} Thurston--Bennequin invariant, Legendrian knots, open books. 

{\it Mathematics Subject Classification:} Primary: 57R17; Secondary: 57M27, 57R65
}\end{NoHyper}}

\medskip

{\scriptsize $^1$Mathematisches Institut, Universit\"at zu K\"oln, Weyertal 86--90, 50931 K\"oln, Germany \\[-2mm]
e-mail: sdurst@math.uni-koeln.de}

\medskip

{\scriptsize $^2$Mathematisches Institut, Universit\"at zu K\"oln, Weyertal 86--90, 50931 K\"oln, Germany \\[-2mm]
e-mail: mkegel@math.uni-koeln.de}

\medskip

{\scriptsize $^3$Institute of Science and Technology Austria, Am Campus 1, 3400 Klosterneuburg, Austria \\[-2mm]
e-mail: mirko.klukas@gmail.com}
\end{center}

\bigskip

\begin{quotation}
\begin{abstract}
We give explicit formulas and algorithms for the computation of the Thurston--Bennequin invariant of a nullhomologous Legendrian knot on a page of a contact open book and on Heegaard surfaces in convex position. Furthermore, we extend the results to rationally nullhomologous knots in arbitrary $3$-manifolds.
\end{abstract}
\end{quotation}

\section{Introduction}

An important class of knots in contact $3$-manifolds is given by \emph{Legendrian knots}, i.e.\ smooth knots tangent to the contact structure. If a Legendrian knot is nullhomologous the so-called \emph{Thurston--Bennequin invariant} compares the contact framing with the Seifert framing. The Thurston--Bennequin invariants of Legendrian knots in a contact $3$-manifold encode a lot of information about the contact structure. For example, a contact structure is overtwisted if and only if there exists a Legendrian unknot with vanishing Thurston--Bennequin invariant. For this and other basic notions in contact geometry we refer the reader to \cite{Geiges2008}.

For Legendrian knots in the unique tight contact structure of the $3$-sphere there is an easy formula to compute the Thurston--Bennequin invariant out of a \emph{front projection} (see Proposition 3.5.9 in \cite{Geiges2008}).
A natural extension is to consider Legendrian knots in \emph{contact surgery diagrams} and to compute their Thurston--Bennequin invariant in the surgered manifold. 
Starting with the work of Lisca, Ozsv{\'a}th, Stipsicz and Szab{\'o} \cite[Lemma 6.6]{MR2557137} several results were obtained in that setting by Geiges and Onaran \cite[Lemma 2]{MR3338830}, Conway \cite[Lemma 6.4]{Conway2014} and Kegel \cite[Section 8]{Kegel2016}. 

In the light of the Giroux correspondence of open books and contact structures (see \cite{Etnyre2004}) another natural way to present a Legendrian knot is to put it on the page of a \emph{compatible open book} of the contact $3$-manifold.
In \cite{Gay2015} Gay and Licata recently generalised the notion of a front projection of a Legendrian knot to this situation. Among other applications is a formula computing the Thurston--Bennequin invariant of a Legendrian knot given in a generalised front projection.

Here we use a more direct approach. The main results are formulas for:

\begin{itemize}
\item[\bf{A.}] The Thurston--Bennequin invariant of nullhomologous Legendrian knots on Heegaard surfaces in convex position in terms of intersection behaviour with the Heegaard curves (see Theorem \ref{thm:tb_surface}),
\item[\bf{B.}]  The Thurston--Bennequin invariant of nullhomologous Legendrian knots on pages of open books in terms of the monodromy (see Theorem \ref{thm:openbook}).
\end{itemize}

We first state and prove the formula for Heegaard surfaces in Section \ref{section:heegaard}, which we then adapt to the setting of open books in Section \ref{section:openbook}. We follow roughly the steps in \cite{Kegel2016}. First we compute the homology of the knot exterior from the Heegaard diagram and then present contact and Seifert framing in this homology. Comparing these two classes then yields the Thurston--Bennequin invariant.
We furthermore present some examples and applications in Section \ref{section:applications} and extend the obtained results to rationally nullhomologous Legendrian knots in Section \ref{rational}.

The other classical invariant of a Legendrian knot, its \emph{rotation number}, can be computed in surgery diagrams by the results mentioned above. 
It would be nice to have a similar formula for the open book setting.

\section{The Thurston--Bennequin invariant in Heegaard diagrams}
\label{section:heegaard}

Let $(M, \xi)$ be a closed $3$-dimensional contact manifold and fix a contact Heegaard splitting $M = V_1 \cup V_2$, i.e.\ a Heegaard splitting such that the Heegaard surface is convex in the sense of Giroux. In particular, the handlebodies $V_1$ and $V_2$ are not assumed to be standard contact handlebodies. Let $K \subset M$ be a Legendrian knot on $\partial V_1 = \partial V_2$ which is nullhomologous in $M$ and intersects the dividing set $\Gamma$ of the convex Heegaard surface $\partial V_1$ transversely. We denote the number of intersection points by $|K \cap \Gamma|$.
Note that for a given knot in a contact manifold it is always possible to find a contact Heegaard splitting such that the knot lies on the Heegaard surface (cf.\ Corollary 4.23 in \cite{Etnyre2004}).

We give a formula to calculate the Thurston--Bennequin invariant of $K$ in this setting.
Let $n$ denote the genus of the Heegaard surface. We may assume that the solid handlebody $V_1$ consists of a single $0$-handle and $n$ $1$-handles and the solid handlebody $V_2$ consists of $n$ $2$-handles and a single $3$-handle.
Let $g_i, g^*_i$, $i = 1, \ldots, n$, be a set of generators of $H_1(\partial V_1;\Z)$ such that the $g^*_i$ are trivial in $H_1(V_1;\Z)$ and $g_i \bullet g^*_j = \delta_{ij}, g_i \bullet g_j = 0 = g^*_i \bullet g^*_j$, where $\bullet$ denotes the intersection product in $H_1(\partial V_1;\Z)$ (see Figures \ref{fig:notation} and \ref{fig:intersection}).
For ease of notation we will not differentiate between an oriented curve and the homology class it represents.
Furthermore, the Heegaard curves on $\partial V_1$, i.e.\ the images of the attaching spheres $c_i$ of the $2$-handles, are called $c'_i$.
We fix orientations of $K$ and of the $c_i$. This is needed for the calculations, but the results are independent of the particular choice.

\begin{figure}[htb] 
\centering
\def\svgwidth{0.9\columnwidth}
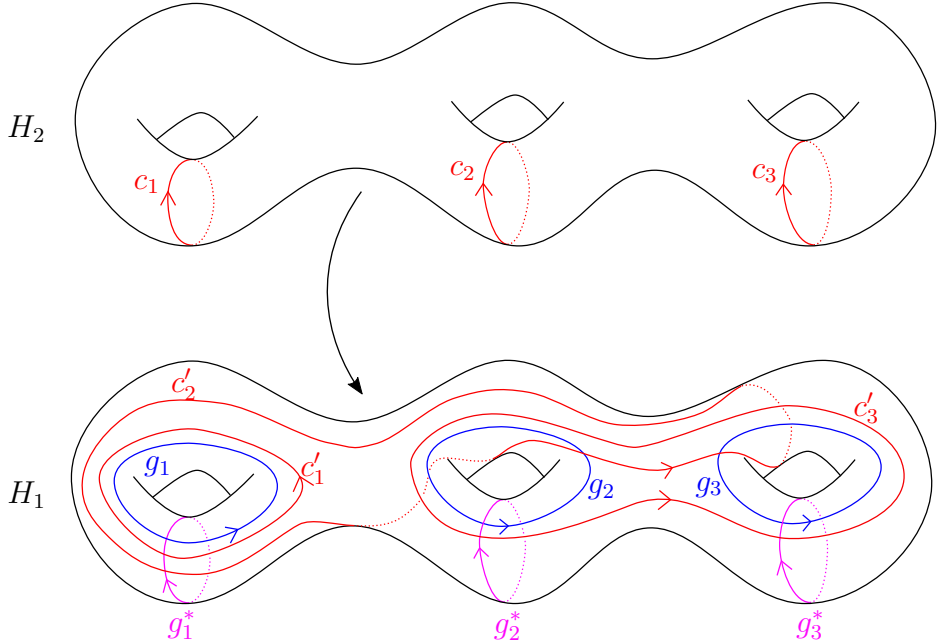
\caption{A Heegaard diagram of $S^1\times S^2$}
\label{fig:notation}
\end{figure}	

\begin{figure}[htb] 
\centering
\def\svgwidth{0.7\columnwidth}
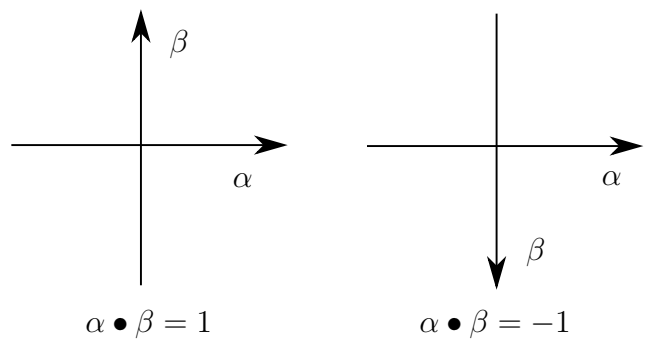
\caption{The intersection pairing in $\mathbb{R}^2$ with standard orientation}
\label{fig:intersection}
\end{figure}

Observe that $H_1(M;\Z)$ is generated by the $g_i$ and there is a relation for every Heegaard curve $c'_j$ (whose expression in terms of the generators can be read off by counting intersections of $c'_j$ with the $g^*_i$, i.e.\ $c'_j = \sum(c'_j\bullet g^*_i) g_i$, cf.\ Chapter 9 in \cite{Rolfsen2003}). Throughout this paper we will omit the coefficient group $\Z$ and all homology groups are understood to be integral if not stated otherwise.
So we have the presentation
$$
H_1(M) = \langle g_1,\ldots, g_n \;|\; c'_1,\ldots, c'_n \rangle.
$$
A knot is nullhomologous in $M$ if and only if its class is a linear combination of the relations in $H_1(M)$ over the integers, i.e. as a class in  $H_1(V_1)$ we can write the nullhomologous knot $K$ as
$$
K = \sum_{i=1}^{n}{E_i c'_i}
$$
for appropriate integers $E_i$.

\begin{theorem}
The Thurston--Bennequin invariant of the Legendrian nullhomologous knot $K$ lying on a Heegaard surface in convex position, transversely intersecting its dividing set, computes as
$$
\tb(K) = -\frac{1}{2}|K \cap \Gamma| + \sum_{i=1}^{n}{E_i \cdot (K \bullet c'_i)}.
$$
\label{thm:tb_surface}
\end{theorem}

\begin{proof}
First we consider the case in which $K$ does not intersect the dividing set $\Gamma$ of the convex Heegaard surface. Then the contact framing of $K$ coincides with the \emph{Heegaard framing},
i.e.\ the framing induced by a parallel copy of $K$ on the Heegaard surface.
We want to use the above presentation of $H_1(M)$ to construct a presentation of $H_1(M\setminus \nu K)$, where $\nu K$ denotes a tubular neighbourhood of $K$ in $M$. To that end we slightly push the curves $g_i$ and $c'_i$ into the handlebody $V_1$ in a neighbourhood of the intersection points with $K$ and denote the resulting curves by $\widetilde{g_i}$ and $\widetilde{c}'_i$ (see Figure \ref{fig:complement}).
Let $\mu$ be a positive meridian of $K$ in $M$. Then $H_1(M\setminus \nu K)$ is generated by $\mu$ together with the $\widetilde{g_i}$ and the relations are $\widetilde{c}'_i - (K \bullet c'_i)\mu$,
$$
H_1(M\setminus \nu K) = \langle \widetilde{g_1},\ldots, \widetilde{g_n}, \mu \;|\; \widetilde{c}'_1 - (K \bullet c'_1)\mu, \, \ldots,\, \widetilde{c}'_n - (K \bullet c'_n)\mu \rangle.
$$

\begin{figure}[htb] 
\centering
\def\svgwidth{0.93\columnwidth}
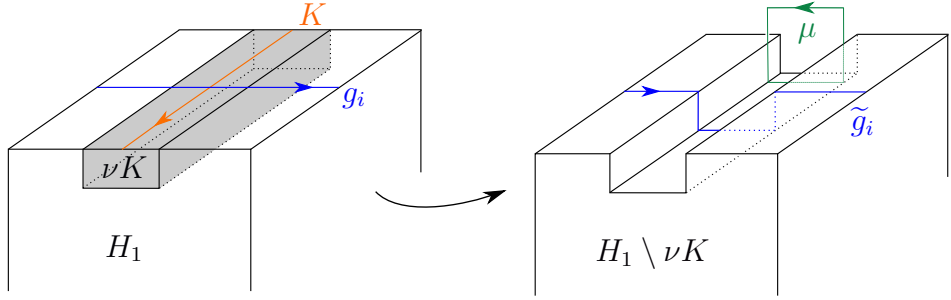
\caption{The relation of the generators in $M$ and $M\setminus \nu K$}
\label{fig:complement}
\end{figure}

Let $\lambda_c$ denote the contact longitude and $\lambda_s$ the Seifert longitude of $K$ (cf.~Section~3.5 in \cite{Geiges2008}). Then the Thurston--Bennequin invariant $\tb(K)$ is defined by the equation
$$\lambda_c = \tb(K) \cdot \mu + \lambda_s$$
in $H_1(\partial\nu K)$.
The Seifert longitude is defined by the condition $\lambda_s = 0$ in $H_1(M\setminus \nu K)$. This yields the equation
$$
-\tb(K) \cdot \mu + \lambda_c = 0 \in H_1(M\setminus \nu K).
$$
In our setting, the contact framing coincides with the Heegaard framing. Therefore the contact longitude $\lambda_c$ is given as a parallel copy of $K$ on the Heegaard surface,
i.e. we have $\lambda_c = K$ in $H_1(\partial V_1)$ and thus $\lambda_c = \sum_{i=1}^{n}{E_i \widetilde{c}'_i}$ in $H_1(M\setminus \nu K)$.
Inserting this expression for the contact longitude into the above equation for the Thurston--Bennequin invariant we get
$$
-\tb(K) \cdot \mu + \sum_{i=1}^{n}{E_i \widetilde{c}'_i} = 0.
$$
Using the relations in $H_1(M\setminus \nu K)$ this transforms to
$$
\tb(K) \mu = \sum_{i=1}^{n}{E_i \widetilde{c}'_i} = \sum_{i=1}^{n}{E_i \widetilde{c}'_i} - \sum_{i=1}^{n}{E_i \big(\widetilde{c}'_i - (K \bullet c'_i)\mu\big)} = \sum_{i=1}^{n}{E_i \cdot (K \bullet c'_i)}\mu.
$$
As the meridian of a nullhomologous knot has infinite order in the knot complement, this proves the first case.

In the general case, when the intersection of $K$ with the dividing set $\Gamma$ is non-empty the result follows from the fact that the contact framing and the framing induced by the Heegaard surface differ by half the number of intersection points of $K$ with the dividing set (cf.\ Theorem 2.30 in \cite{Etnyre}).
\end{proof}

\begin{algorithm}
\label{rem:algorithm}\emph{Computing the Thurston--Bennequin invariant.}
Using the formula from Theorem \ref{thm:tb_surface} we can compute the Thurston--Bennequin invariant of a Legendrian nullhomologous knot lying on a convex Heegaard surface algorithmically.
Define vectors
$$
A := \left( K \bullet g^*_i \right)_{i=1,\ldots,n}
$$
and
$$
I := \left( K \bullet c'_i \right)_{i=1,\ldots,n}
$$
and a matrix
$$
C := \left( c'_j \bullet g^*_i \right)_{i,j=1,\ldots,n}.
$$
Solve the equation
$$
A = C \cdot E
$$
over the integers (such a solution exists exactly if $K$ is nullhomologous). Then the Thurston--Bennequin invariant is given by:
$$
\tb = -\frac{1}{2}|K \cap \Gamma| + \langle E, I \rangle.
$$
\end{algorithm}

\begin{example}
We compute the first homology group of the manifold $M$ given by Figure \ref{fig:notation} as
$$
H_1(M) = \langle g_1, g_2, g_3 \;|\; c'_1, c'_2, c'_3 \rangle = \langle g_1, g_2, g_3 \;|\; g_1, g_1, g_2 + g_3 \rangle \cong \Z,
$$
where we use $c'_j = \sum(c'_j\bullet g^*_i) g_i$. In fact, one can show that $M$ is diffeomorphic to $S^1 \times S^2$.

\begin{figure}[htb] 
\centering
\def\svgwidth{0.9\columnwidth}
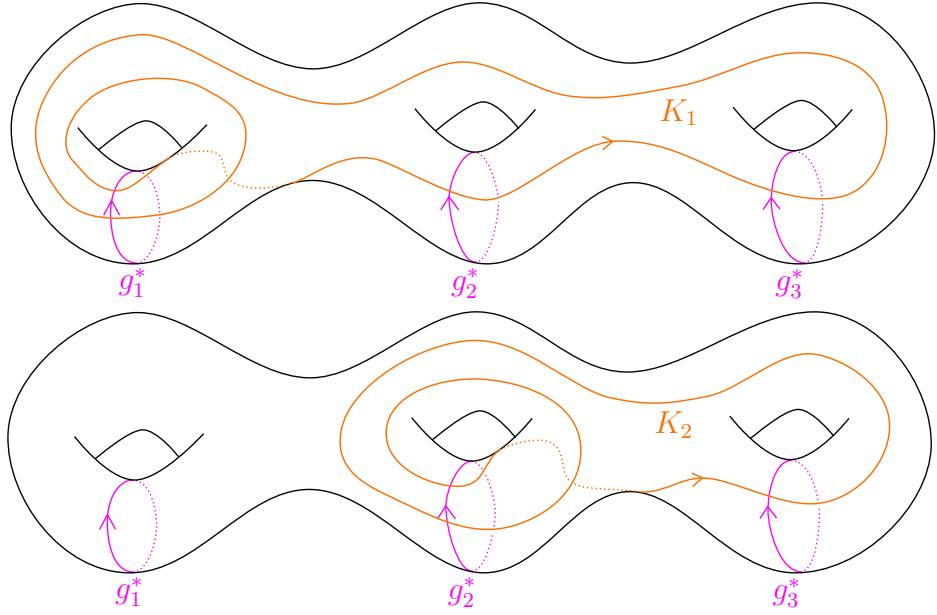
\caption{Knots on a Heegaard surface of $S^1\times S^2$}
\label{fig:heegaardknot}
\end{figure}

Now consider two knots $K_1$ and $K_2$ in $M$ as shown in Figure \ref{fig:heegaardknot}. 
The matrix $C$ is equal to
$$
C = \left( c'_j \bullet g^*_i \right)_{i,j=1,\ldots,3} = \begin{pmatrix} 1 & 1 & 0 \\ 0 & 0 & 1\\ 0 & 0 & 1 \end{pmatrix}
$$
and the knots are encoded by
$$
A_1 = \left( K_1 \bullet g^*_i \right)_{i=1,\ldots,3} = ( 2, 1, 1 )^\intercal.
$$
and
$$
A_2 = \left( K_2 \bullet g^*_i \right)_{i=1,\ldots,3} = ( 0, 2, 1 )^\intercal.
$$
The equation $A_1 = CE$ admits integral solutions, e.g.\ $( 1, 1, 1 )^\intercal$, which means $K_1$ is nullhomologous. However, $A_2 = CE$ is not solvable at all, so $K_2$ is not nullhomologous.
\end{example}

\section{The Thurston--Bennequin invariant in open books}
\label{section:openbook}

In this section we use the result on Heegaard surfaces to give a computable formula for the Thurston--Bennequin invariant of a nullhomologous Legendrian knot on the page of an open book and furthermore a way to check whether a knot on a page is nullhomologous.
Note that it is always possible to find an open book supporting the contact structure such that a given Legendrian knot lies on a page (see Corollary 4.23 in \cite{Etnyre2004}).
Let $(S, \phi = \T^{\varepsilon_l}_{l} \circ \cdots \circ \T^{\varepsilon_1}_{1} )$ be a contact open book with monodromy $\phi$ encoded by a concatenation of Dehn twists. Here $\T^{\varepsilon_k}_{k}$ denotes a Dehn twist along the curve $T_k$ with sign $\varepsilon_k$. Let $(M,\xi)$ be the resulting contact manifold.
Choose an arc basis $a_i$, $i=1,\ldots,n$, i.e.\ a system of arcs such that $S$ becomes a disk when cutting along them, in such a way that the arcs meet the curves $T_k$ transversely.
Using the intersection product on $S$ we define a matrix $C$ via
\begin{equation*}
\resizebox{1 \textwidth}{!} 
{
$c_{ij} := \sum\limits_{m=1}^{l}\; \sum\limits_{1\leq k_1 < \ldots < k_m \leq l }{ \varepsilon_{k_1} \cdots \varepsilon_{k_m} (T_{k_m} \bullet T_{k_{m-1}}) \cdots (T_{k_2} \bullet T_{k_1}) (T_{k_1} \bullet a_j) (T_{k_m} \bullet a_i) }.$
}
\end{equation*}

\begin{theorem}
\label{thm:openbook}
Let $(S, \phi = \T^{\varepsilon_l}_{l} \circ \cdots \circ \T^{\varepsilon_1}_{1} )$ be a contact open book with monodromy $\phi$ encoded by a concatenation of Dehn twists and fixed arc basis $a_i$, $i=1,\ldots,n$ of $S$ as above. Let $K$ be a Legendrian knot on $S$. Define a vector $A$ by 
$
A = \left( K \bullet a_i \right)_{i=1,\ldots,n}.
$
\begin{enumerate}
\item
$K$ is nullhomologous if and only if there exists an integer solution $E$ of
$$
A = C \cdot E.
$$
\item
If $K$ is nullhomologous its Thurston--Bennequin invariant is equal to
$$
\tb (K) = - \langle E, A \rangle.
$$
\end{enumerate}
\end{theorem}

\begin{proof}
Note that if $K$ is a Legendrian knot on $S$ its contact framing coincides with the framing induced by the page $S$ (cf.\ Lemma 3.5 in \cite{Etnyre2004}).
With the chosen arc basis $a_i$, $i=1,\ldots,n$, we get that
$$
( \Sigma := S_{1} \cup S_{2} ,\; g^*_i := (a_i)_{1} \cup (a_i)_{2},\; c'_i :=  (a_i)_{1} \cup (\phi(a_i))_{2} )
$$
is a genus $n =  (2 \cdot \textrm{genus}(S) + r - 1)$ Heegaard diagram for $M$, where $r$ is the number of boundary components of $S$ (this is a slight variation of the classic approach published in \cite{MR2577470}). Here $S_1$ and $S_2$ are two copies of the page $S$, with the orientation of $S_2$ reversed, glued along their boundary, i.e.\ the Heegaard surface $\Sigma$ is the double of $S$, and $(a_i)_{j}$ denotes a copy of $a_i$ on $S_j$ (see Figure \ref{fig:booktoheegaard}). Curves and arcs on $S_2$ are always assumed to be oriented oppositely to their counterparts on $S$ to give rise to oriented curves on $\Sigma$.
Furthermore, the curves $g^*_i$ and $c'_i$ are understood to be slightly isotoped to only have transverse intersections. We identify $S$ with $S_1$, so the knot $K$ lies on $S_1$.

\begin{figure}[htb] 
\centering
\def\svgwidth{0,9\columnwidth}
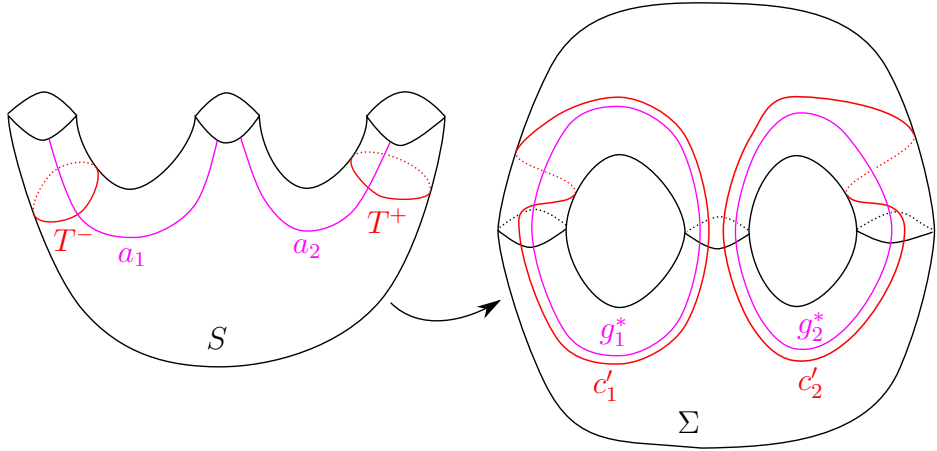
\caption{From an open book decomposition to a Heegaard diagram}
\label{fig:booktoheegaard}
\end{figure}

Having transformed the open book into a Heegaard diagram, Theorem \ref{thm:tb_surface} gives a formula for computing the Thurston--Bennequin invariant. In particular, we will use Algorithm~\ref{rem:algorithm} and adapt it such that it only uses input data from the open book, i.e.\ $\tb$ is computable without constructing a Heegaard diagram first.
We have
$$
A = \left( K \bullet g^*_i \right)_{i=1,\ldots,n} = \left( K \bullet ((a_i)_{1} \cup (a_i)_{2}) \right)_{i=1,\ldots,n} = \left( K \bullet a_i \right)_{i=1,\ldots,n},
$$
where the last equality arises from restriction to $S_1$ as $K$ lies only on $S_1$.
Analogously, the matrix $C$ has entries
$$
c_{ij} = c'_j \bullet g^*_i = ( (a_j)_{1} \cup (\phi(a_j))_{2} ) \bullet ( (a_i)_{1} \cup (a_i)_{2} ) = \phi(a_j) \bullet a_i,
$$
where the last term is again read in $S_2$ and comes from restriction (we isotope the curves such that there are no intersection points on the boundary and consider the algebraic intersection number).
Observe that in our current setting we have 
$$
I = \left( K \bullet c'_i \right)_{i=1,\ldots,n} = A
$$
by a similar argument.

As we have shown in Section \ref{fig:complement} the knot $K$ is nullhomologous if and only if the equation
$$
A = C \cdot E
$$
has an integral solution $E$, and in that case the Thurston--Bennequin invariant computes as
$$
\tb = \langle E, A \rangle.
$$
It remains to calculate the entries of the matrix $C$ in terms of the Dehn twists $\T^{\varepsilon_l}_{l} \circ \cdots \circ \T^{\varepsilon_1}_{1}$ encoding the monodromy.
Let $\alpha$ be any curve on a surface and $T^{\varepsilon}$ a Dehn twist. Then the homology class of the image of $\alpha$ under $T^{\varepsilon}$ is
$$
\alpha + \varepsilon ( T \bullet \alpha ) T,
$$
where we identify curves with their classes as usual.
Repeatedly applying this to the $a_j$ yields
\begin{equation*}
\resizebox{1 \textwidth}{!} 
{
$\phi(a_j) = a_j + \sum\limits_{m=1}^{l} \; \sum\limits_{1\leq k_1 < \ldots < k_m \leq l }{ \varepsilon_{k_1} \cdots \varepsilon_{k_m} (T_{k_m} \bullet T_{k_{m-1}}) \cdots (T_{k_2} \bullet T_{k_1}) (T_{k_1} \bullet a_j) T_{k_m} }$
}
\end{equation*}
and thus
\begin{equation*}
\resizebox{1 \textwidth}{!} 
{
$c_{ij} = \sum\limits_{m=1}^{l} \; \sum\limits_{1\leq k_1 < \ldots < k_m \leq l }{ \varepsilon_{k_1} \cdots \varepsilon_{k_m} (T_{k_m} \bullet T_{k_{m-1}}) \cdots (T_{k_2} \bullet T_{k_1}) (T_{k_1} \bullet a_j) (T_{k_m} \bullet a_i) },$
}
\end{equation*}
where we use the intersection product on $S_2$.
In applications, however, we want to consider intersections on the page $S$,
which has the opposite orientation. This provides for the negative sign in the
formula to compute the Thurston--Bennequin invariant in an open book, i.e.\ we have
$$\tb = - \langle E, A \rangle.$$
\end{proof}

\begin{remark}
Note that in the case of disjoint Dehn twist curves $T_k$ the expression of the matrix entries $c_{ij}$ reduces to
$$
c_{ij} = \sum_{k=1}^{l}{ \varepsilon_{k} (T_{k} \bullet a_j) (T_{k} \bullet a_i) }.
$$
In particular, $C$ is symmetric.
\end{remark}

\begin{remark}
\label{remark:non-unique}
The particular choice of a solution $E$ does not impact the result since two different solutions differ by a vector $K$ in the kernel of $C$ and $A$ is in the image of $C$. Thus, the scalar product of $A$ and $K$ vanishes, see also Example \ref{ex:non-unique}.
\end{remark}

\section{Applications and Examples}
\label{section:applications}

\begin{example}
\emph{Unknot in the standard $3$-sphere.}
Consider the open book decomposition of $(S^3, \xi_\textrm{st})$ with page $S$ an annulus and the monodromy given by a positive Dehn twist $T^{+}$ along the central curve $T$ and let $K$ be a Legendrian knot parallel to $T$ on the page $S$. In this example an arc basis of $S$ consists of a single arc $a$ only, which we choose to be a linear segment joining the boundary components of the annulus.

\begin{figure}[htb] 
\centering
\def\svgwidth{0,9\columnwidth}
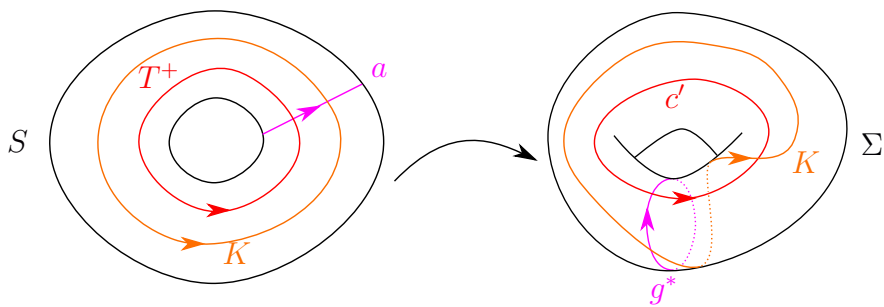
\caption{The Legendrian unknot in  $(S^3, \xi_\textrm{st})$}
\label{fig:sphere}
\end{figure}	

Choosing orientations as depicted in Figure \ref{fig:sphere} we get 
$$
A =  K \bullet a = -1
$$
and
$$
C = \varepsilon ( T \bullet a )^2 = 1 \cdot (-1)^2 = 1.
$$
The knot $K$ is nullhomologous since the equation $-1 = 1 \cdot E$ has the solution $E = -1$. This we knew before since any knot in $S^3$ is nullhomologous, but we need a particular solution $E$ to calculate $\tb$.
We then compute the Thurston--Bennequin invariant as
$$
\tb (K) = - \langle E, A \rangle = -1 \cdot (-1) \cdot (-1) = -1.
$$
The Heegaard diagram  on the right hand side of Figure \ref{fig:sphere} encodes the same situation. Here it becomes clear that $K$ is the unknot.
This particular Heegaard splitting arises from the open book picture on the left by performing a Dehn twist.
\end{example}

\begin{example}
\emph{Unknot in an overtwisted $3$-sphere.}
We change the monodromy in the previous example to be a negative Dehn twist $T^{-}$ along $T$. As above, we then have $A = -1$ but $C$ becomes
$$
C = \varepsilon ( T \bullet a )^2 = -1 \cdot (-1)^2 = -1
$$
and $E=1$ solves $A=CE$. So we get
$$
\tb (K) = - \langle E, A \rangle = -1 \cdot 1 \cdot (-1) = 1.
$$
Stabilising $K$ once yields an overtwisted disc, so the contact structure is indeed overtwisted.
\end{example}

\begin{example}
\label{ex:non-unique}

\begin{figure}[htb] 
\centering
\def\svgwidth{0,6\columnwidth}
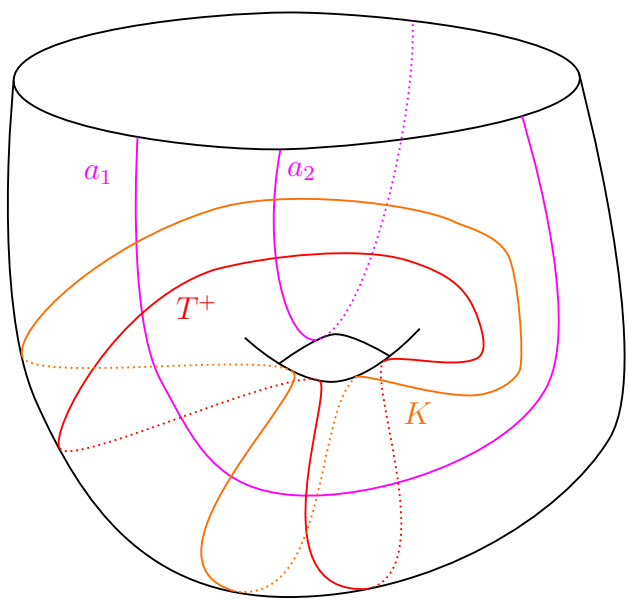
\caption{A nullhomologous knot $K$ with non-unique $E$}
\label{fig:nonunique}
\end{figure}

Consider the open book for $(S^1 \times S^2, \xi_{st})$ depicted in Figure~\ref{fig:nonunique}.
We have
$$
A = \begin{pmatrix} 2\\ 1 \end{pmatrix}
$$
and
$$
C = \begin{pmatrix} 4 & 2\\ 2&  1 \end{pmatrix}.
$$
The equation $A=CE$ is solvable over the integers, so $K$ is nullhomologous. However, the solution is non-unique. Solutions are of the form
$$
E_n = \begin{pmatrix} 2\\ 1 \end{pmatrix} + n \begin{pmatrix} 1\\ -2 \end{pmatrix}
$$
for $n\in\Z$.
Then
$$
\tb (K) = - \langle E_n, A \rangle = -1  \cdot (-1) \cdot (-1) = -1,
$$
i.e.\ the result is independent of the chosen solution $E_n$. This is always the case (see Remark \ref{remark:non-unique}).
\end{example}

\begin{example}
\label{ex:stabilisation}
\emph{Stabilisations.}
Let $K$ be a nullhomologous Legendrian knot on the page $S$ of an open book $(S, \phi = \T^{\varepsilon_l}_{l} \circ \cdots \circ \T^{\varepsilon_1}_{1} )$. We want to compute the Thurston--Bennequin invariant of the stabilised knot $K_\textrm{stab}$ in the stabilised open book $(S_\textrm{stab}, \phi_\textrm{stab} = \T^{\varepsilon_{l+1}}_{l+1} \circ \T^{\varepsilon_l}_{l} \circ \cdots \circ \T^{\varepsilon_1}_{1} )$.
Let $A, C, E$ be the data associated to the original open book and knot. With an additional arc $a$ and orientations chosen as in Figure \ref{fig:stabilisation}, we have 
$$
A_\textrm{stab} = \begin{pmatrix} A\\ 1 \end{pmatrix}
$$
and
$$
C_\textrm{stab} = \begin{pmatrix} C & 0 \\ 0 & \varepsilon_{l+1} \end{pmatrix}
$$
since $\T_{l+1}$ is disjoint from the other Dehn twists.
The equation $A_\textrm{stab} = C_\textrm{stab} E_\textrm{stab}$ is then solved by the integral vector
$$
E_\textrm{stab} = \begin{pmatrix} E\\ \varepsilon_{l+1} \end{pmatrix}
$$
and we compute $\tb$ to be
\begin{equation*}
\resizebox{1 \textwidth}{!} 
{
$\tb (K_\textrm{stab}) = - \langle E_\textrm{stab}, A_\textrm{stab} \rangle = - \big\langle \begin{pmatrix} E\\ \varepsilon_{l+1} \end{pmatrix}, \begin{pmatrix} A\\ 1 \end{pmatrix} \big\rangle = - \langle E, A \rangle -\varepsilon_{l+1} = \tb (K) -\varepsilon_{l+1}.$
}
\end{equation*}

\begin{figure}[htb] 
\centering
\def\svgwidth{0,9\columnwidth}
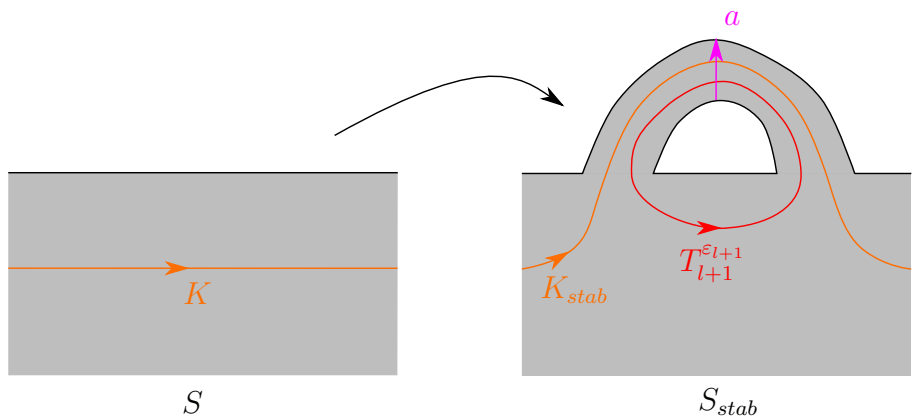
\caption{A stabilisation of $K$ obtained by a positive stabilisation of the open book}
\label{fig:stabilisation}
\end{figure}

\end{example}

\section{Rationally nullhomologous knots}
\label{rational}

In this section we study  rationally nullhomologous Legendrian knots as proposed in Baker-Grigsby \cite{MR2552000}, Baker-Etnyre \cite{MR2884030} and Geiges-Onaran \cite{MR3338830}.
In particular, we generalise Theorems \ref{thm:tb_surface} and \ref{thm:openbook} to rationally nullhomologous Legendrian knots.
Let $K$ be a knot in $M$. We call $K$ \emph{rationally nullhomologous} if its homology class is of finite order $d>0$ in $H_1(M)$.
Let $\nu K$  be a tubular neighbourhood of $K$ and denote the meridian by $\mu \subset \partial\nu K$.

\begin{definition}
A \emph{Seifert framing} of a rationally nullhomologous knot $K$ of order $d$ is a class $r \in H_1(\partial \nu K)$ such that
\begin{itemize}
\item $\mu \bullet r = d$,
\item $r = 0$ in $H_1(M \setminus \nu K)$.
\end{itemize}
\end{definition}

It is obvious that every rationally nullhomologous knot has a Seifert framing;
uniqueness however is not obvious.

\begin{lemma}
The Seifert framing of a rationally nullhomologous knot is unique.
\end{lemma}

\begin{proof}
Let $r_1$ and $r_2$ be Seifert framings. Let $\mu, \lambda$ be an oriented basis of $H_1(\partial \nu K)$, where $\mu$ is represented by a meridian of $K$.
Then we can write
$$
r_i = p_i \mu + q_i \lambda.
$$ 
As $r_i$ is a Seifert framing we have $q_i = d$ with $d$ the order of $K$.
The classes $r_1$ and $r_2$ are equal if considered in $H_1(M \setminus \nu K)$. Therefore we have $p_1 \mu = p_2 \mu$ in $H_1(M \setminus \nu K)$.
But a meridian of $K$ intersects a rational Seifert surface non-trivially, so $\mu$ cannot be a torsion element. Hence $p_1 = p_2$, i.e.\ the framings coincide.
\end{proof}

Existence and uniqueness of the Seifert framing enables us to define a rational Thurston--Bennequin invariant, which coincides with the usual definition in the nullhomologous case, and is well-defined in arbitrary contact $3$-manifolds.

\begin{definition}
The \emph{rational Thurston--Bennequin invariant} of a rationally nullhomologous Legendrian knot $K$ is defined as
$$
\tb_\Q (K) = \frac{1}{d} \left( \lambda_c \bullet r \right)
$$
where $\lambda_c$ denotes the contact longitude and $r$ the Seifert framing, and the intersection is taken in $H_1(\partial \nu K)$.
\end{definition}

Observe that this means that we have the equality
$$
r = d\lambda_c - d \tb_\Q (K) \mu
$$
in $H_1(\partial \nu K)$.

Now consider a Legendrian knot $K$ on a convex Heegaard surface not intersecting the dividing set. Using the notation from Section \ref{section:heegaard}, such a knot is rationally nullhomologous of order $d$ in $M$ if and only if the equation
$$
dA = C \cdot E
$$
admits a solution $E$ over the integers and $d$ is the minimal natural number for which a solution exists.
In that case, fix a solution $E$.
Analogously to the nullhomologous case we then have
$$
d \tb_\Q(K) \mu = \sum_{i=1}^{n}{E_i \widetilde{c}'_i} = \sum_{i=1}^{n}{E_i \cdot (K \bullet c'_i)}\mu
$$
in $H_1(M \setminus \nu K)$.
Since $\mu$ has infinite order we thus proved the following theorem.

\begin{theorem}
The rational Thurston--Bennequin invariant of the Legendrian rationally nullhomologous knot $K$ of order $d$ lying on a convex Heegaard surface, transversely intersecting its dividing set $\Gamma$, computes as
$$
\tb_\Q(K) = -\frac{1}{2}|K \cap \Gamma| + \frac{1}{d} \sum_{i=1}^{n}{E_i \cdot (K \bullet c'_i)}
= -\frac{1}{2}|K \cap \Gamma| + \frac{1}{d} \langle E, I \rangle .
$$
\end{theorem}

Similarly, Theorem \ref{thm:openbook} generalises to the result stated below.

\begin{theorem}
Let $(S, \phi = \T^{\varepsilon_l}_{l} \circ \cdots \circ \T^{\varepsilon_1}_{1} )$ be a contact open book with monodromy $\phi$ encoded by a concatenation of Dehn twists and fixed arc basis $a_i$, $i=1,\ldots, n,$ of $S$. Let $K$ be a Legendrian knot on $S$. Define a vector $A$ by 
$
A = \left( K \bullet a_i \right)_{i=1,\ldots,n}.
$
\begin{enumerate}
\item
$K$ is rationally nullhomologous of order $d$ if and only if there exists an integer solution $E$ of
$$
dA = C \cdot E
$$
and $d$ is the minimal natural number for which a solution exists.
\item
If $K$ is rationally nullhomologous of order $d$ its rational Thurston--Bennequin invariant is equal to
$$
\tb_\Q (K) = - \frac{1}{d} \langle E, A \rangle.
$$
\end{enumerate}
\end{theorem}

\section*{Acknowledgements}
The authors are very grateful to Hansj\"org Geiges for fruitful discussions and advice and Christian Evers for helpful remarks on a draft version.


\begin{thebibliography}{99}

\bibitem{MR2884030} K.~Baker and J.~Etnyre, Rational linking and contact geometry, {\it Perspectives in analysis, geometry, and topology, Progr.~Math.}, {\bf 296} Birkh\"auser/Springer, New York (2012), 19--37.

\bibitem{MR2552000} K.~Baker and J.~Grigsby, Grid diagrams and Legendrian lens space links, {\it J.~Symplectic Geom.}, {\bf 7} (2009), 415--448.

\bibitem{Conway2014} J.~Conway, Transverse Surgery on Knots in Contact $3$-Manifolds, arXiv:1409.7077.

\bibitem{Etnyre} J.~Etnyre, Convex surfaces in contact geometry: class notes, Lecture notes available on: \url{http://people.math.gatech.edu/~etnyre/preprints/papers/surfaces.pdf}

\bibitem{Etnyre2004} J.~Etnyre, Lectures on open book decompositions and contact structures, arXiv:math/0409402.

\bibitem{Gay2015} D.~Gay and J.~Licata, Morse structures on open books, arXiv:1508.05307.

\bibitem{Geiges2008} H.~Geiges, An Introduction to Contact Topology, {\it Cambridge University Press} (2008).

\bibitem{MR3338830} H.~Geiges and S.~Onaran, Legendrian rational unknots in lens spaces, {\it J.~Symplectic Geom.}, {\bf 13} (2015), 17--50.

\bibitem{MR2577470} K.~Honda, W.~Kazez and G.~Mati{\'c}, On the contact class in Heegaard Floer homology, {\it J.~Differential Geom.}, {\bf 83} (2009), 289--311.

\bibitem{Kegel2016} M.~Kegel, The Legendrian Knot Complement Problem, arXiv:1604.05196.

\bibitem{MR2557137} P.~Lisca, P.~Ozsv{\'a}th, A.~Stipsicz and Z.~Szab{\'o}, Heegaard Floer invariants of Legendrian knots in contact three-manifolds, {\it J.~Eur.~Math.~Soc.~(JEMS)}, {\bf 11} (2009), 1307--1363.

\bibitem{Ozbagci2011} B.~Ozbagci, Contact handle decompositions, {\it Topology Appl.}, {\bf 158} (2011), 718--727.

\bibitem{Ozbagci2004} B.~Ozbagci and A.~Stipsicz, Surgery on Contact $3$-Manifolds and Stein Surfaces, {\it Springer} (2004), Bolyai Society mathematical studies.

\bibitem{Rolfsen2003} D.~Rolfsen, Knots and Links, {\it AMS Chelsea Pub.} (2003).


\end{thebibliography}
\end{document}